\documentclass[a4paper,10pt]{amsart}
\usepackage{amssymb}
\usepackage{amsmath}
\usepackage{amsthm}
\usepackage{amsfonts}
\usepackage{enumerate}
\usepackage{hyperref}

\usepackage[textwidth=14cm,centering]{geometry}

\newcommand{\abar}{{\ensuremath{\bar{a}}}}
\newcommand{\bbar}{{\ensuremath{\bar{b}}}}

\DeclareMathOperator{\Th}{Th}  

\DeclareMathOperator{\rk}{rk}

\DeclareMathOperator{\td}{td}  

\newcommand{\alg}{\ensuremath{\mathrm{alg}}} 

\DeclareMathOperator{\ldim}{ldim}  
\DeclareMathOperator{\mrk}{mrk}  
\DeclareMathOperator{\Mat}{Mat}  

\newcommand{\Z}{\ensuremath{\mathbb{Z}}}
\newcommand{\Q}{\ensuremath{\mathbb{Q}}}

\newcommand{\C}{\ensuremath{\mathbb{C}}}

\newcommand{\Cexp}{\ensuremath{\mathbb{C}_{\mathrm{exp}}}}

\newcommand{\ga}{\ensuremath{\mathbb{G}_\mathrm{a}}}   
\newcommand{\gm}{\ensuremath{\mathbb{G}_\mathrm{m}}}  

\renewcommand{\phi}{\varphi}

\renewcommand{\ge}{\ensuremath{\geqslant}}
\newcommand{\tuple}[1]{\ensuremath{\langle #1 \rangle}}
\newcommand{\class}[2]{\ensuremath{\left\{ #1 \,\left|\, #2 \right.\right\}}}

\newcommand{\subs}{\subseteq} 

\newcommand{\elsubs}{\preccurlyeq} 

\newcommand{\minus}{\ensuremath{\smallsetminus}}

\newcommand{\nstrong}{\ensuremath{\not\kern-4pt\lhd\;}} 

\newcommand{\gen}[1]{\ensuremath{\left\langle #1 \right\rangle}} 

\newcommand{\cross}{\ensuremath{\times}}

\newbox\noforkbox \newdimen\forklinewidth
\forklinewidth=0.3pt \setbox0\hbox{$\textstyle\smile$}
\setbox1\hbox to \wd0{\hfil\vrule width \forklinewidth depth-2pt
  height 10pt \hfil}
\wd1=0 cm \setbox\noforkbox\hbox{\lower 2pt\box1\lower
2pt\box0\relax}
\def\unionstick{\mathop{\copy\noforkbox}\limits}

\def\nonfork_#1{\unionstick_{\textstyle #1}}

\setbox0\hbox{$\textstyle\smile$} \setbox1\hbox to \wd0{\hfil{\sl
/\/}\hfil} \setbox2\hbox to \wd0{\hfil\vrule height 10pt depth
-2pt width
                \forklinewidth\hfil}
\newbox\doesforkbox
\setbox\doesforkbox\hbox{\lower 2pt\box1 \lower
2pt\box2\lower2pt\box0\relax}
\def\nunionstick{\mathop{\copy\doesforkbox}\limits}

\def\fork_#1{\nunionstick_{\textstyle #1}}

\newcommand{\findep}[4]{\ensuremath{#1 \nonfork_{#3}^{#4} #2}}

\newcommand{\algindep}[3]{\findep{#1}{#2}{#3}{\mathrm{ACF}}}

\newcommand{\leteq}{\mathrel{\mathop:}=}

\theoremstyle{plain}
\newtheorem{theorem}{Theorem}
\newtheorem{lemma}[theorem]{Lemma}
\newtheorem{prop}[theorem]{Proposition}

\newtheorem{cor}[theorem]{Corollary}

\theoremstyle{definition}
\newtheorem{defn}[theorem]{Definition}

\theoremstyle{remark}

\providecommand{\Cexp}{\mathbb{C}_{\exp}}

\providecommand{\C}{\mathbb{C}}

\providecommand{\Q}{\mathbb{Q}}
\providecommand{\Z}{\mathbb{Z}}

\newcommand{\CZ}{\C_\Z}

\newcommand{\B}{\ensuremath{\mathbb{B}}}
\newcommand{\ECF}{\ensuremath{\mathbf{ECF}}} 

\newcommand{\sstrong}{\ensuremath{\mathrel{\prec \kern -0.53ex \shortmid}}}

\newcommand{\good}{good} 

\newcommand{\doi}[1]{\href{http://dx.doi.org/#1}{doi:#1}}

\title[$\Q$ is not $\forall$-definable in $\B$]{The rational field is not universally definable in pseudo-exponentiation}
\author{Jonathan Kirby}
\date{\today \\ \indent Supported by EPSRC grant EP/L006375/1}

\address{School of Mathematics\\
University of East Anglia\\
Norwich Research Park\\
Norwich\\
NR4~7TJ\\
UK}

\email{jonathan.kirby@uea.ac.uk}

\begin{document}

\begin{abstract}
We show that the field of rational numbers is not definable by a universal formula in Zilber's pseudo-exponential field.
\end{abstract}

\maketitle

Boris Zilber's pseudo-exponential field $\tuple{\B,+,\cdot,-,0,1,\exp}$ is conjecturally isomorphic to the complex exponential field $\Cexp = \tuple{\C;+,\cdot,-,0,1,\exp}$ \cite{Zilber05peACF0}. While $\Cexp$ is defined analytically, $\B$ is constructed entirely by algebraic and model-theoretic methods, and for example it does not have a canonical topology. The conjecture that they are isomorphic contains Schanuel's conjecture of transcendental number theory, so seems out of reach of current methods. However, it is interesting to ask what properties known to hold of one of the structures can be proved to hold of the other, and often this sheds new light on both structures.

A structure $M$ is \emph{model complete} if and only if every definable subset of $M^n$ is definable by an existential formula. Equivalently, every definable subset is defined by a universal formula, or equivalently again, whenever $M_1$ and $M_2$ are both elementarily equivalent to $M$ and $M_1 \subs M_2$, then $M_1 \elsubs M_2$.

The rational field $\Q$ is definable both in $\Cexp$ and in $\B$ by the existential formula
\[\exists y_1 \exists y_2 [e^{y_1} = 1 \wedge e^{y_2} = 1 \wedge x \cdot y_1 = y_2 \wedge y_1 \neq 0]\]
which states that $x$ is a ratio of kernel elements. (As usual, we write $e^a$ to mean $\exp(a)$.) We write $Q(M)$ for the subset of a model $M$ defined by this formula. We also write $\ker(M)$ for the subset defined by $e^x = 1$, and $Z(M)$ for the subset defined by
$\forall y[e^y = 1 \to e^{xy} = 1]$. We have $Z(\Cexp) = Z(\B) = \Z$, the standard integers, and $\ker(\B) = \tau Z(\B)$ for a transcendental number $\tau$ (corresponding to $2\pi i$ in $\C$). Laczkovich showed that $\Z$ is also definable by an existential formula in $\Cexp$ \cite{Laczkovich03}, and the same formula works in $\B$ \cite{KMO12}, so $\Z$ is not a counterexample to model completeness.

Marker \cite{Marker06} gave a topological proof that $\Q$ is not definable by a universal formula in $\Cexp$, thus proving that $\Cexp$ is not model complete. Macintyre asked whether $\B$ is model complete \cite{Macintyre_Antalya_talk}. I answered this negatively \cite{FPEF} by finding a pair of models $M_1 \subs M_2$ of the first-order theory of $\B$ with $M_1 \not\elsubs M_2$. However, that paper only dealt with models of the theory of $\B$ with standard kernel, that is, $\ker(M) = \tau \Z$,  and the definable set shown to be existentially definable but not universally definable has nothing to do with the rationals. The paper \cite{ECFCIT} extends the methods of \cite{FPEF} to the case of models of $\Th(\B)$ with non-standard kernel. Using these extended methods, this note proves:
\begin{theorem}\label{main theorem}
The rational field $\Q$ is not definable by a universal formula in $\B$.
\end{theorem}

The proof goes by constructing exponential fields $F \subs M$, both elementarily equivalent to \B, and an element $q \in F$ such that $q \in Q(M)$ but $q \notin Q(F)$. That shows that $\Q$ cannot be universally definable in \B. The construction of the element $q$ is somewhat separate from the construction of $F$ and $M$, and also demonstrates that $\Q$ is not universally definable in the structure $\C_\Z$ which consists of the complex field expanded by a predicate for the integers. This result for $\C_\Z$ can also be deduced from Marker's result above, but we give an explicit proof in section~1 below. Section~2 contains the necessary background about exponentially closed fields, and the proof of Theorem~\ref{main theorem} forms section~3.

\section{The complex field with a predicate for the integers}

Write $\CZ$ for the structure $\tuple{\C;+,\cdot,-,0,1,\Z}$, the complex field with a predicate $Z$ naming the integers. The rational field $\Q$ is defined in $\CZ$ by the formula
\begin{equation*}
\tag{$*$} \exists y_1 \exists y_2 [Z(y_1) \wedge Z(y_2) \wedge x \cdot y_1 = y_2 \wedge y_1 \neq 0]
\end{equation*}
and we write $Q(M)$ for the realisation of this formula in any model $M$ of $\Th(\CZ)$. 
\begin{prop}\label{q}
There is an elementary extension $M$ of $\CZ$ with an element $q \in Q(M)$ such that $q$ is transcendental, but $\Q(q)^\alg \cap Z(M) = \Z$.
\end{prop}
From the proposition, we can deduce quickly:
\begin{cor}
The rational field $\Q$ is not universally definable in $\CZ$.
\end{cor}
\begin{proof}
Given $M$ and $q$ as in the proposition, let $F = \Q(q)^\alg$, considered as a substructure of $M$. Then $Z(F) = \Z$, and so $F \models \Th(\CZ)$ because models of this theory are just algebraically closed field extensions of a ring $Z$ elementarily equivalent to $\Z$ \cite{Vozoris_thesis}. But $q \in (F \cap Q(M)) \minus Q(F)$ so there is no universal formula defining $\Q$.
\end{proof}

\begin{proof}[Proof of Proposition~\ref{q}]
We consider the type $p(y)$ given by the formula $(*)$ stating that $y$ is rational together with the formulas
\[\class{f(y) \neq 0}{f \in \Z[Y] \minus \{0\}}\]
which collectively say that $y$ is transcendental and the formulas 
\[\class{\phi_g(y) }{g \in \Z[Y,W], \text{ irreducible over $\Z$, } \frac{\partial g}{\partial Y} \neq 0,  \frac{\partial g}{\partial W} \neq 0}\]
where $\phi_g(y) $ is the formula $\forall w [g(y,w) = 0 \to \neg Z(w)]$.

Then if $M$ is a model of $\Th(\CZ)$ and $M \models p(q)$, we have $q \in Q(M)$, transcendental. Furthermore if $a \in \Q(q)^\alg \cap Z(M)$ then either $a$ is algebraic, in which case $a \in \Z$ because $M \models \Th(\CZ)$, or there is $g \in \Z[Y,W]$ irreducible over $\Z$ with both partial derivatives non-zero and $g(q,a) = 0$, witnessing the algebraic dependence between $a$ and $q$. Then $a \notin Z(M)$ because $M \models \phi_g(q)$, a contradiction. So $\Q(q)^\alg \cap Z(M) = \Z$ as required.

So it is enough to show that $p(y)$ is consistent, which we do by showing any finite subtype is realised in the standard model $\CZ$. 

Fix a real transcendental number $y_0$. We claim that for each $g \in \Z[Y,W]$ which is irreducible over $\Z$ and such that $\frac{\partial g}{\partial Y}$ and $\frac{\partial g}{\partial W}$ are nonzero, there is a neighbourhood $U_g$ of $y_0$ in $\C$ such that for any $y \in U_g$, $\CZ \models \phi_g(y)$.

For such a $g$, let $h(W) = g(y_0,W)$. Then $\frac{dh}{dW} = \frac{\partial g}{\partial W}(y_0,W)$ which is nonzero because $\frac{\partial g}{\partial W}$ is nonzero and $y_0$ is transcendental, so $h(W)$ is a non-constant polynomial and hence has zeros $w_1,\ldots,w_d$ in $\C$, where $d$ is the degree of $h$. Since $g$ is irreducible over $\Z$ we have $h$ irreducible over $\Z[y_0]$ and so the $w_i$ are distinct  and it follows that $\frac{dh}{dW}(w_i) \neq 0$, that is, $\frac{\partial g}{\partial W}(y_0,w_i) \neq 0$ for each $i$.

We apply the complex implicit function theorem to the polynomial $g(Y,W)$ at each point $(y_0,w_i)$ to find a neighbourhood $U_g$ of $y_0$ in $\C$, disjoint neighbourhoods $V_i$ of $w_i$ in $\C$ and analytic functions $s_i : U_g \to V_i$ such that $s_i(y_0) = w_i$, and for all $y \in U_g$ and each $i=1,\ldots, d$ we have $g(y,s_i(y)) = 0$, and the only solution $w$ in $V_i$ to $g(y,w) = 0$ is $s_i(y)$. Since for each $y \in U_g$ the polynomial $g(y,W)$ has degree (at most) $d$ in $W$, these must be the only solutions $w$ in $\C$ to $g(y,w) = 0$.

If some $w_i$ were in $\Z$ (or even algebraic) then since $y_0$ is transcendental and $g(y_0,w_i) = 0$ we must have $g(Y,w_i) = 0$. Then $W-w_i$ would be a factor of $g(Y,W)$ so, since $g(Y,W)$ is irreducible, we get $g(Y,W) = \pm(W-w_i)$. Then $\frac{\partial g}{\partial Y}$ vanishes, a  contradiction. So no $w_i$ is in $\Z$. Since all the functions $s_i$ are continuous and $\Z$ is discrete, we can shrink $U_g$ to ensure that for all $y \in U_g$ and each $i=1,\ldots,d$ we have $s_i(y) \notin \Z$. That proves the claim.

Now let $p_0$ be a finite subtype of $p$ and let $U = \bigcap_{\phi_g \in p_0} U_g$. Then $U$ is open and contains the real point $y_0$, so $\Q \cap U$ is infinite. Choose $q \in \Q \cap U$ satisfying all the finitely many conditions $f(y) \neq 0$ from $p_0$. So $p_0$ is consistent and, by compactness, so is $p$.
\end{proof}

\section{Exponentially closed fields}

We consider structures $\tuple{M;+,\cdot,-,0,1, \exp}$ in the language  of rings expanded by a unary function symbol $\exp$, satisfying some or all of the following list of axioms, which are numbered as in \cite{ECFCIT}.
\begin{description}
 \item[1. ELA-field] $M$ is an algebraically closed field of characteristic zero, and its exponential map $\exp$ is a homomorphism from its additive group to its multiplicative group, which is surjective.
\end{description}

Any model of axiom 1 is called an \emph{ELA-field}.
\begin{description}
\item[2. Standard kernel] the kernel of the exponential map is an infinite cyclic group generated by a transcendental element $\tau$.
\end{description}

Since standard kernel is not preserved under elementary extensions, we also consider the following weaker version of axiom 2 which is.
\begin{description}
\item[2$'$] There is $\tau \in M$, transcendental over $Z(M)$, such that $\ker(M) = \class{\tau z}{z \in Z(M)}$. Furthermore, $\tuple{Z(M);+,\cdot,-,0,1}$ is a model of the full first-order theory of the ring of standard integers.
\end{description}

For the last two axioms we need some more notation and terminology. By $\td(Y/X)$ we mean the transcendence degree of the field extension $\Q(XY)/\Q(X)$ and by $\ldim_\Q(Y/X)$ we mean the dimension of the $\Q$-vector space spanned by $X \cup Y$, quotiented by the subspace spanned by $X$.
If $X$, $Y$ are subsets of the multiplicative group $\gm(M)$, we write $\mrk(Y/X)$ for the \emph{multiplicative rank}, that is, the $\Q$-linear dimension of the divisible subgroup spanned by $X\cup Y$, quotiented by the divisible subgroup spanned by $X$ and all the torsion. 

Let $V$ be a subvariety of $\ga^n(M)\cross \gm^n(M)$ and let $(\abar,\bbar)$ be a point in $V$, generic over $M$. Then $V$ is said to be \emph{additively free} if $\ldim_\Q(\abar/M) = n$, and \emph{multiplicatively free} if $\mrk(\bbar/\gm(M)) = n$. $V$ is \emph{rotund} if for every matrix $L \in \Mat_{n \cross n}(\Z)$, we have $\td(L \abar, \bbar^L/M) \ge \rk L$, where $\rk L$ means the rank of the matrix $L$, and $\bbar^L$ is just the usual matrix action as a linear map but in the multiplicative group rather than the additive group.
\begin{description}
  \item[3$'$. The Schanuel Property over the kernel] \ The predimension function 
 \[ \Delta(\bbar) \leteq \td(\bbar, \exp(\bbar) /  \ker(M)) - \ldim_\Q(\bbar /  \ker(M)) \]
 satisfies $\Delta(\bbar) \ge 0$ for all tuples $\bbar$ from $M$.

\item[4. Strong exponential-algebraic closedness] If $V$ is a rotund, additively and multiplicatively free subvariety of $\ga^n(M) \cross \gm^n(M)$ defined over $M$ and of dimension $n$, and $\abar$ is a finite tuple from $M$, then there is $\bbar \in M^n$ such that $(\bbar,e^\bbar) \in V$ and is generic in $V$ over $\abar$.
\end{description}
Axiom $3'$ puts a strong restriction on what systems of exponential polynomial equations can have solutions in $M$, based on Schanuel's conjecture. Axiom 4 is a suitable form of existential closedness, the content of which is that any system of equations which has a solution in an extension of $M$ which does not violate axiom $3'$ already has a solution in $M$.

\begin{defn}
The class \ECF\ of \emph{exponentially closed fields} is defined to be the class of models of axioms 1, $2'$, $3'$ and 4. 
\end{defn}
If the diophantine conjecture CIT is true, \ECF\ is exactly the class of all models elementarily equivalent to $\B$  \cite[Theorem~1.3]{ECFCIT}. However we do not need to rely on CIT as unconditionally all models in \ECF\ are elementarily equivalent to \B.

We will make use of a strengthening of axiom 4.
\begin{defn}
A model $M \in \ECF$ is said to be \emph{saturated over its kernel} if whenever $V$ is as in axiom~4 
and $A$ is a subset of $M$ with $|A| < |M|$, then there is $\bbar$ in $M$ such that $(\bbar,e^\bbar) \in V$ and is generic in $V$ over $A$, and also the \emph{exponential transcendence degree} of $M$ is equal to $|M|$.
\end{defn}
We will not make use of the exponential transcendence degree so we do not give the definition. The main theorem of \cite{ECFCIT} states that such models exist in large enough cardinalities, and are unique once the model of the ring of integers is specified.
\begin{theorem}[{\cite[Theorem~1.1]{ECFCIT}}]
For each $\aleph_0$-saturated model $R$ of $\Th(\Z)$, and for each cardinal $\lambda > 2^{\aleph_0}$ with $\lambda \ge |R|$, there is exactly one model $M \in \ECF$ such that $Z(M) = R$ and such that $M$ is saturated over its kernel.
\end{theorem}

For the rest of this note we fix an $M \in \ECF$ such that $Z(M)$ is an $\aleph_0$-saturated model of $\Th(\Z)$ and $M$ is saturated over its kernel. The kernel generator described in axiom 2$'$ is defined only up to $\pm$, so we choose one of them to be $\tau$.


We need a little more notation. For subsets $A$, $B$, and $C$ of $M$ we write $\algindep{A}{B}{C}$ to mean that $A$ is independent from $B$ over $C$ in the sense of algebraically closed fields, that is, every finite tuple $\abar \in A$ satisfies $\td(\abar/B \cup C) = \td(\abar/C)$. 

By $A^\alg$, we mean the field-theoretic algebraic closure of $A$ in $M$ and we write $\gen{A}$ for the $\Q$-linear span of $A$ in $M$. 

\begin{defn}
We say that $A$ is \emph{semistrong} in $M$ and write $A \sstrong M$ if 
\begin{enumerate}[(i)]
\item for every finite tuple $\bbar$ from $M$, the relative predimension function 
\[\Delta(\bbar/A) \leteq \td(\bbar, \exp(\bbar) /  \ker(M),A,\exp(A)) - \ldim_\Q(\bbar / \ker(M),A)\]
satisfies $\Delta(\bbar/A) \ge 0$; and 
\item $\algindep{A,\exp(A)}{\ker(M)}{\ker(M)\cap \gen{A}}$.
\end{enumerate}
\end{defn}
Note that $\ldim_\Q(\bbar / \ker(M),A) = \mrk(\exp(\bbar)/\exp(A))$, so we can also write the relative predimension function as 
$\Delta(\bbar/A) = \td(\bbar, \exp(\bbar) /  \ker(M),A,\exp(A)) - \mrk(\exp(\bbar) /\exp(A))$.

If $B = \gen{A,\bbar}$ we also write $\Delta(B/A)$ for $\Delta(\bbar/A)$, and if $A = \gen{\abar}$ we write $\Delta(\bbar/\abar)$ for $\Delta(\bbar/A)$.
The addition property for $\Delta$ is easily verified: for all $A, \abar, \bbar$,
\[\Delta(\abar\bbar/A) = \Delta(\bbar/A\abar) + \Delta(\abar/A).\]

In the paper \cite{ECFCIT}, much use is made of partial exponential fields, and the semistrong property is defined for them. Here we will work inside the fixed model $M$, so it is equivalent and notationally simpler to work with $\Q$-linear subspaces.

\section{The Proof of Theorem~\ref{main theorem}}

Since $Z(M)$ is $\aleph_0$-saturated, there are $r_1,r_2 \in Z(M)$ such that, taking $q = r_1/r_2$, $q$ satisfies the type $p(y)$ from the proof of Proposition~\ref{q}. It is easy to check that $r_1$ and $r_2$ are algebraically independent over $\Q$ (that is, they do not satisfy any non-trivial polynomial equations with standard rational coefficients). Indeed, otherwise they would lie in $\Q(q)^\alg$, but then the type $p$ implies they are both standard integers, contradicting the transcendence of $q$.

So $q \in Q(M)$, but $q$ is not in the $\Q$-linear span of $Z(M)$. Let $p_1 = \tau r_1$ and $p_2 = \tau r_2$, so $q = p_1/p_2$ and $p_1,p_2 \in \ker(M)$. Then $p_1$ and $p_2$ are algebraically independent over the kernel generator $\tau$ because $\tau$ is transcendental over $Z(M)$ by axiom $2'$.

We will build $F$ as the union of a chain of $\Q$-linear subspaces of $M$. At each stage we need certain conditions to hold to ensure that we do not run into problems later. We capture these conditions in the next definition.
\begin{defn}
Let $A$ be a $\Q$-linear subspace of $M$ such that $\tau, q \in A$. Then $A$ is \emph{\good} (\emph{for the purpose of this proof}) if 
\begin{enumerate}
\item $(A \cup \exp(A))^\alg \cap \ker(M) = \tau \Z$, so in particular $A \cap \ker(M) = \tau\Z$ and $p_1,p_2 \notin A$;
\item $\gen{A,p_1,p_2} \sstrong M$; and
\item $|A| < |M|$.
\end{enumerate}
\end{defn}
Condition (2) splits into the clauses (i) and (ii) of the definition of semistrongness above. Clause (i) does not depend on $p_1$ and $p_2$ at all, but clause (ii) does, since it says (given condition (1)) that the only algebraic dependencies between $A\cup \exp(A)$ and the kernel of $M$ are witnessed by $\{\tau, p_1,p_2\}$, that is, that $\algindep{A,\exp(A)}{\ker(M)}{\{\tau,p_1,p_2\}}$.

We will start the chain with $A_0 = \gen{\tau,q}$. 
\begin{lemma}\label{A_0 good}
$A_0$ is good.
\end{lemma}
\begin{proof}
First we observe using the  Schanuel property over the kernel that $\Delta(q) \ge 0$, that is,
\begin{equation*}
\tag{$\dagger$}  \td(q,e^q/\ker(M)) - \ldim_\Q(q/\ker(M)) \ge 0.
\end{equation*}
Now $q$ is algebraic over $\ker(M)$ because $q=p_1/p_2$, but $q$ is not in the $\Q$-linear span of the kernel, so $(\dagger)$ reduces to $\td(e^q/\ker(M)) = 1$, and it follows that $\Delta(q) = 0$. Also $\Delta(\tau,q) = \Delta(q) = 0$, so for any $\bbar \in M$,
\[\Delta(\bbar/\tau,q) = \Delta(\bbar,\tau,q) - \Delta(\tau,q) = \Delta(\bbar,\tau,q) \ge 0\]
by the addition property for $\Delta$ and the Schanuel Property over the kernel. So clause (2)(i) holds.

We have shown that $e^q$ is transcendental over $\ker(M)$, and we have $\tau,q \in \{\tau,p_1,p_2\}^\alg \subs \ker(M)^\alg$, so it follows that 
$\quad \algindep{\{\tau,q,e^q\}}{\ker(M)}{\{\tau,p_1,p_2\}} \quad $ which is clause (2)(ii).

For clause (1), we note that $(A_0 \cup \exp(A_0))^\alg = \{\tau,q,e^q\}^\alg$. Suppose $a \in \{\tau,q,e^q\}^\alg \cap \ker(M)$. Then $a,\tau,q \in \ker(M)^\alg$ but $e^q$ is transcendental over $\ker(M)$, so in particular $e^q \notin \{a,\tau,q\}^\alg$. By the exchange property for algebraic closure, $a \in \{\tau,q\}^\alg$, so we have
\[(A_0 \cup \exp(A_0))^\alg \cap \ker(M) = \{\tau,q\}^\alg \cap \ker(M).\]

Suppose $a \in \{\tau,q\}^\alg \cap \ker(M)$ and let $x = a / \tau$, so $x \in \{\tau,q\}^\alg \cap Z(M)$. Then $x,q \in Z(M)^\alg$ but $\tau$ is transcendental over $Z(M)$ by axiom 2$'$ so, again by exchange, $x \in \Q(q)^\alg \cap Z(M)$. Then since $q$ satisfies the type $p$ we have $x \in \Z$. So (1) holds.

Finally, $|A_0| = \aleph_0 < |M|$. So clause (3) holds and thus $A_0$ is good.
\end{proof}

\begin{lemma}\label{ELA good}
If $A$ is a \good\ $\Q$-linear subspace of $M$ then there is an ELA-subfield $K$ of $M$ containing $A$ such that $|K| = |A|$ and $K$ is also \good.
\end{lemma}
\begin{proof}
The union of a chain  of length $ < |M|$ of \good\ subspaces of $M$ is still \good\ because conditions (1) and (2) have finite character, so it is enough to show that, given $a \in (A\cup\exp(A))^\alg$, nonzero, there is a \good\ subspace $A_2$ of $M$ containing $A$ such that $a \in A_2 \cap \exp(A_2)$.

First let $A_1 = \gen{A,a}$. Since $\gen{A,p_1,p_2} \sstrong M$ we have

\[\td(e^a /\ker(M),A,\exp(A)) = \td(a,e^a/\ker(M),A,\exp(A)) \ge \ldim_\Q(a/\ker(M),A).\]

Since $a \in (A\cup\exp(A))^\alg$ and $(A\cup\exp(A))^\alg \cap \ker(M) = \tau \Z \subs A$ it follows that $\ldim_\Q(a/\ker(M),A) = \ldim_\Q(a/A)$. So $\td(e^a/\ker(M),A,\exp(A)) \ge \ldim_\Q(a/A)$. If $a \in A$ we have $A_1 = A$ and we are done. Otherwise $\ldim_\Q(a/A) = 1$ so $\td(e^a/\ker(M),A,\exp(A)) = 1$, and $\Delta(a/A) = \Delta(a/A,p_1,p_2) = 1-1 = 0$.

Thus $(A_1\cup\exp(A_1))^\alg \cap \ker(M) = (A \cup \exp(A))^\alg \cap \ker(M) = \tau\Z$.  
If $\bbar$ is a tuple from $M$ then 
\begin{eqnarray*}
\Delta(\bbar/A_1,p_1,p_2) &=& \Delta(\bbar/A,a,p_1,p_2)\\
&=& \Delta(a\bbar/A,p_1,p_2) - \Delta(a/A,p_1,p_2) \\
& =& \Delta(a\bbar/A,p_1,p_2) - 0\\
&\ge& 0
\end{eqnarray*}
because $\gen{A,p_1,p_2} \sstrong M$. Thus $\gen{A_1,p_1,p_2}_M \sstrong M$. Clearly $|A_1| = |A| < |M|$, so $A_1$ is good.

If $a \in \exp(A_1)$ then set $A_2 = A_1$ and we are done. Otherwise, choose any $c \in M$ such that $e^c = a$ and set $A_2 = \gen{A_1,c}$. Then we use the same argument as above, with $A_1$ in place of $A$ and swapping the roles of the additive and multiplicative sides, to show that $A_2$ is good. In detail,
\begin{eqnarray*}
\td(c/\ker(M),A_1,\exp(A_1)) & = & \td(c,e^c/\ker(M),A_1,\exp(A_1))\\
& \ge & \ldim_\Q(c/\ker(M),A_1) \\
&=& \mrk(a/\exp(A_1)) \\
&=& 1
\end{eqnarray*}
so $c$ is transcendental over $\ker(M) \cup A_1 \cup \exp(A_1)$, and $\Delta(A_2,p_1,p_2/A_1,p_1,p_2) = 0$ so, by the same argument as above, $\gen{A_2,p_1,p_2} \sstrong M$. Hence $A_2$ is good.
\end{proof}

\begin{lemma}\label{V good}
Suppose $K$ is a \good\ ELA-subfield of $M$ and $V$ is a rotund, additively and multiplicatively free subvariety of $\ga^n(M) \cross \gm^n(M)$, defined over $K$ and of dimension $n$. Then there is a \good\ ELA-extension field $K_V$ of $K$ inside $M$ such that there is $(\abar,e^\abar) \in V(K_V)$, generic in $V$ over $K$, and $|K_V| = |K|$.
\end{lemma}
\begin{proof}
Since $M$ is saturated over its kernel there is $\abar \in M^n$ such that $(\abar,e^\abar) \in V(M)$,  generic in $V$ over $K \cup\{p_1,p_2\}$. We have $\gen{K,p_1,p_2}_M \sstrong M$, so $\td(\abar,e^\abar/\ker(M), K) \ge \mrk(e^\abar/K)$. Since $V$ is multiplicatively free and $(\abar,e^\abar)$ is generic in $V$ over $K$ we have $\mrk (e^\abar / K) = n$, and so $\td(\abar,e^\abar/\ker(M),K) = n$.

Let $H = \gen{K,\abar)}_M$ and $H' = \gen{K,\abar,p_1,p_2}_M$. Then $\algindep{H}{\ker(M)}{K}$, so $H^\alg \cap \ker(M) = k^\alg \cap \ker(M) = \tau\Z$. Also $H'^\alg \cap\ker(M) = K^\alg \cap \ker(M)$, so $\algindep{H}{K,p_1,p_2}{K}$. We also know $\algindep{K,p_1,p_2}{\ker(M)}{\{\tau, p_1,p_2\}}$, so $\algindep{H}{\ker(M)}{\{\tau, p_1,p_2\}}$. Also $\Delta(H'/K,p_1,p_2) = 0$, so $H' \sstrong M$. Thus $H$ is good. Applying Lemma~\ref{ELA good} we can take $K_V$ to be some \good\ ELA-extension of $H$ in $M$.
\end{proof}

\begin{prop}
There is $F \subs M$ containing $\tau$ and $q$ such that $F \in \ECF$ and $q \notin Q(F)$.
\end{prop}
\begin{proof}
By Lemma~\ref{A_0 good} $A_0$ is \good, so applying Lemma~\ref{ELA good} there is a countable \good\ ELA-subfield, $F_1$ of $M$. Now enumerate all the rotund, additively and multiplicatively free subvarieties defined over $F_1$, and apply Lemma~\ref{V good} in turn for each and iterate, noting that the union of a chain of \good\ ELA-subfields of $M$ is still a \good\ ELA-subfield. At stage $\omega^2$ we get an ELA-subfield $F$ of $M$ which is strongly exponentially-algebraically closed. It satisfies the Schanuel property over the kernel, since every exponential subfield of $M$ does. Since $F$ is \good\ it has standard kernel. Hence $F \in \ECF$ and, by construction, $q \in F$.

Since $F$ has standard kernel, $Q(F) = \Q$. The element $q$ is transcendental, so is not in $Q(F)$.
\end{proof}

That completes the proof of Theorem~\ref{main theorem}.

\subsection*{Acknowledgements}
I would like to thank David Marker for asking me the question of whether $\Q$ is universally definable in $\B$ when I gave a talk at the Mathematical Sciences Research Institute in Berkeley, California, during the program on Model Theory, Arithmetic Geometry and Number Theory in Spring 2014, which was supported by the National Science Foundation under Grant No.~0932078~000. Most of the work was also done during that program. This work is also supported by EPSRC grant EP/L006375/1. Thanks are also due to the anonymous referee, who made several suggestions to improve the presentation and noticed a couple of places where I had oversimplified the argument.

\end{document}